\documentclass[11pt,english]{amsart}
\usepackage[T1]{fontenc}
\usepackage[latin9]{inputenc}
\usepackage{babel}
\usepackage{amsthm}
\usepackage{amstext}
\usepackage{amssymb}
\usepackage[unicode=true,pdfusetitle,
 bookmarks=true,bookmarksnumbered=false,bookmarksopen=false,
 breaklinks=false,pdfborder={0 0 1},backref=false,colorlinks=false]
 {hyperref}
\usepackage{breakurl}

\makeatletter
\numberwithin{equation}{section}
\numberwithin{figure}{section}
\theoremstyle{plain}
\newtheorem{thm}{\protect\theoremname}
  \theoremstyle{plain}
  \newtheorem{lem}[thm]{\protect\lemmaname}
  \theoremstyle{plain}
  \newtheorem{prop}[thm]{\protect\propositionname}
  \theoremstyle{definition}
  \newtheorem{defn}[thm]{\protect\definitionname}

\usepackage{amsfonts}\usepackage{amsthm}\usepackage{amscd}\usepackage[all]{xy}

\newcommand{\ra}{\rightarrow}

\newcommand{\ov}{\overline}

\newcommand{\cO}{{\mathcal O}}

\newcommand{\cX}{{\mathcal X}}
\newcommand{\cY}{{\mathcal Y}}

\newcommand{\bN}{{\mathbb N}}

\newcommand{\bR}{{\mathbb R}}

\newcommand{\bZ}{{\mathbb Z}}
\newcommand{\bQ}{{\mathbb Q}}
\newcommand{\bF}{{\mathbb F}}

\newcommand{\bA}{{\mathbb A}}
\newcommand{\bP}{{\mathbb P}}

\newcommand{\GL}{\operatorname{GL}}

\makeatother

  \providecommand{\definitionname}{Definition}
  \providecommand{\lemmaname}{Lemma}
  \providecommand{\propositionname}{Proposition}
\providecommand{\theoremname}{Theorem}

\begin{document}

\title{Homogeneous number of free generators}

\author{Menny Aka}

\author{Tsachik Gelander}

\author{Gregory A. So{\u\i}fer}
\begin{abstract}
We address two questions of Simon Thomas. First, we show that for
any $n\geq3$ one can find a four generated free subgroup of $SL_{n}(\bZ)$
which is profinitely dense. More generally, we show that an arithmetic
group $\Gamma$ which admits the congruence subgroup property, has
a profinitely dense free subgroup with an explicit bound of its rank.
Next, we show that the set of profinitely dense, locally free subgroups
of such an arithmetic group $\Gamma$ is uncountable.
\end{abstract}
\maketitle

\section{Introduction}

Let $G$ be a simply-connected semisimple algebraic group defined
over $\bQ$ with fixed embedding to $GL_{n}$. Let $\Gamma$ be an
arithmetic subgroup of $G(\bar{\bQ})$, i.e., a group commensurable
to $G(\bZ):=G(\bar{\bQ})\cap GL_{n}(\bZ)$. Assume moreover that $G(\bR)$
is non-compact. The aim of this paper is to show the following: 
\begin{thm}
\label{thm:Main Theorem}Assume that $G$ admits the congruence subgroup
property (see $\S$\ref{sub:CSP}), and let $d(\widehat{\Gamma})$
be the minimal number of generators of the profinite completion of
$\Gamma$ (see $\S$\ref{sub:Basic-properties-profinite}). Then,
there exists a free subgroup $F\subset\Gamma$ on at most $2+d(\widehat{\Gamma})$
generators which is profinitely dense, i.e., maps onto any finite
quotient of $\Gamma$. 
\end{thm}
Let $\alpha(\Gamma)$ be the minimal rank of a profinitely dense free
subgroup of $\Gamma$. That is, Theorem \ref{thm:Main Theorem} claim
that $\alpha(\Gamma)\leq d(\Gamma)+2$ . In \cite{SV2000} it is proved
that $\Gamma$ as above admits a profinitely dense free subgroup of
finite rank. Consequently, Simon Thomas asked whether one can find
a uniform bound on $\alpha(SL_{n}(\bZ)),n\geq3$. It is known that
$SL_{n}(\bZ)$ is generated by two elements for all $n\ge2$ (see
\cite{Tr62}) and that $SL_{n}(\bZ)$ (see $\S$\ref{sub:CSP}) admits
the Congruence Subgroup Property for $n\geq3$. Thus Theorem \ref{thm:Main Theorem}
implies that for any $n\geq3$ there exists a free profinitely dense
subgroup of $SL_{n}(\bZ)$ with rank $\leq4$. In particular, given
a family $\{\Gamma_{n}\}$ for such arithmetic groups a uniform bound
on $d(\Gamma_{n})$ will provide a uniform bound on $\alpha(\Gamma_{n})$.
It is interesting whether $\alpha(SL_{n}(\bZ))=2$ for all $n$. We
note that there are arithmetically defined families such that $\alpha(\Gamma_{n})$
is not uniformly bounded. For example, let $\Delta_{n}:=SL_{n}(\bZ)$
and for a rational prime $p$ let $\Delta_{n}(p):=\{\gamma\in\Delta_{n}:\gamma\equiv I(mod\, p)\}$.
It is shown in \cite[Theorem 1.1]{LS76} that the finite quotient
$\Delta_{n}(p)/\Delta_{n}(p^{2})$ is a vector space over $\bF_{p}$
of dimension $n^{2}-1$ so any profinitely dense and free subgroup
has as least $n^{2}-1$ generators. That is $\alpha(\Delta_{n}(p))\geq n^{2}-1$
and in particular not uniformly bounded.

We remark that although it is shown in \cite{BG2007} that the profinite
completion $\widehat{\Gamma}$ of any non-virtually solvable subgroup
$\Gamma$, has a free dense subgroup of rank $d(\Gamma)$, it is in
general impossible to find a free group inside $\Gamma$. For example,
Fuchcian groups are LERF \cite{SC78}(i.e. every finitely generated
subgroup is the intersection of the finite-index subgroups that contain
it), hence cannot admit a finitely generated profinitely dense proper
subgroup. 

We also address another question of Simon Thomas. Let $\mathfrak{U}$
be the set of locally free subgroups of $\Gamma$ containing a profinitely
dense finitely generated subgroup.
\begin{thm}
\label{thm:Main Theorem maximal subgroups}The set $\mathfrak{U}_{m}$
of maximal elements of $\mathfrak{U}$ is uncountable.
\end{thm}

In essence, these theorems can be proved using techniques and results
of \cite{MS81,BG2007}. The advantage of the following scheme of proof
is in being elementary, using simple tools from ergodic theory.\\
Following Tits \cite{Tits1972}, in order to find free subgroups we
use dynamics on projective spaces and we review relevant definitions
and properties in $\S$\ref{sec:Dynamics-on-projective}. Tits' original
result allows us to find a Zariski-dense free subgroup $\langle h_{1},h_{2}\rangle$
of $\Gamma$ of rank $2$. Assuming the Congruence Subgroup Property,
the closure of $\langle h_{1},h_{2}\rangle$ in the profinite topology
is of the form $\widehat{\Gamma'}$ for a  finite-index subgroup $\Gamma'<\Gamma$,
as we explain in $\S$\ref{sec:Profinite-completions}. In order to
find a profinitely dense free subgroup, we will find so-called {}``ping-pong''
partners to $h_{1}$ and $h_{2}$ that will belong to specified cosets
of $\Gamma'$ in $\Gamma$. We will need to add at most $d(\widehat{\Gamma})$
elements in order to construct a profinitely dense free subgroup.
This will be done in two steps. The first step, is to find elements
in cosets with prescribed dynamics on the projective space. In this
step we establish the main new technique of this paper; we use the
mixing property of the action of $G$ on the homogeneous space $G/\Gamma$.
This is done in $\S$\ref{sec:Elements-in-cosets} where we also recall
necessary notions from Ergodic Theory. The second step, is to use,
with necessary modifications, the mechanism of rooted free system
(which originates in \cite{MS81}) in order to inductively add {}``ping-pong''
partners with desirable properties. This is done in $\S$\ref{sec:Free-rooted-systems}.
Using all the above ingredients, we conclude the proofs of Theorems
\ref{thm:Main Theorem} and \ref{thm:Main Theorem maximal subgroups}
in $\S$\ref{sec:Proof-main-thm}.

\section{Dynamics on projective spaces\label{sec:Dynamics-on-projective}}

\subsection{Proximal and Hyperbolic elements}

Let $V$ be a vector space of dimension $n$ over a local field $K$,
$\bP(V)$ the associated projective space and $v\mapsto[v]$ the associated
projection map. The group $GL(V)$ acts naturally on $\bP(V)$ by
$g[v]=[gv]$. An element $g\in\GL(V)$ is called \emph{hyperbolic}
if it is semisimple and admits a unique (counting multiplicities)
eigenvalue of maximal absolute value and minimal absolute value. We
denote by $\mathfrak{H}(G)$ the set of hyperbolic elements of $G$. 

For $g\in\mathfrak{H}(G)$, let $\{v_{1},\dots,v_{n}\}$ be a basis
of eigenvectors such that $v_{1}$ correspond to the unique maximal
eigenvalue of $g$ and $v_{n}$ correspond to the unique minimal eigenvalue
of $g$. Note that although $g$ is not necessarily diagonalizable
over $K$, $span(v_{1},\dots,v_{n-1})$ and $span(v_{2},\dots,v_{n})$
are defined over $K$. Indeed, this follows since there is a unique
extension of the norm on $K$ to any algebraic extension of $K$ (See
e.g. \cite[Proposition 6.4]{Neukirch}).

We denote the following subsets of $\bP(V)$ by
\[
A^{+}(g)=[v_{1}],A^{-}(g)=[v_{n}]
\]
\[
B^{+}(g)=[span(v_{2},\dots,v_{n})],B^{-}(g)=[span(v_{1},\dots,v_{n-1})]
\]

and $A^{\pm}(g)=A^{+}(g)\cup A^{-}(g),B^{\pm}(g)=B^{+}(g)\cup B^{-}(g)$.
Note that $A^{\pm}(h)\subset B^{\pm}(h)$. 

For further use we record some basic properties: 
\[
hA^{+}(g)=A^{+}(hgh^{-1})
\]
 and likewise for $A^{-},B^{+},B^{-}$. Also 
\begin{equation}
\forall g\in\mathfrak{H}(G)\quad A^{\pm}(g)=A^{\pm}(g^{n}),B^{\pm}(g)=B^{\pm}(g^{n}).\label{eq:attractig same for powers}
\end{equation}

and 
\begin{equation}
\forall g\in\mathfrak{H}(G)\quad A^{+}(g^{-1})=A^{-}(g)\label{eq:inverse}
\end{equation}

\subsection{Distance on projective spaces}

In order to study separation properties of projective transformations,
we consider the following, so-called \emph{standard metric} on $\bP(V)$.
For any $[v],[w]\in\bP(V)$ let 
\[
d(a,b)=\frac{\Vert v\wedge w\Vert}{\Vert v\Vert\Vert w\Vert}.
\]
See \cite[Section 3]{BG2003} for the choices of the above norms and
related properties. The reader should check that $d$ is well-defined
and that $d$ induces the topology on $\bP(V)$ that is inherited
from the local field $K$. For any two sets $A,B\subset\bP(V)$ we
set $d(A,B)=min\{d(a,b):a\in A,b\in B\}.$ We also need the following
notion of distance between sets. For any two sets $A,B\subset\bP(V)$
we define the Hausdorff distance 
\[
d_{h}(A,B)=max\{\underset{a\in A}{sup}\,\underset{b\in B}{inf}d(a,b),\underset{b\in B}{sup}\,\underset{a\in A}{inf}d(a,b)\}
\]
The only property of the Hausdorff distance that we are going to use
is the following: Let $a\in\bP(V),B,C\subset\bP(V)$. Then, $d(a,C)\leq d(a,B)+d_{h}(B,C)$.

\subsection{Transversality}

We call two elements $g,h\in\mathfrak{H}(G)$ \emph{transversal,}
and denote $g\perp h$,\emph{ }if $A^{\pm}(g)\subset\bP\setminus B^{\pm}(h)$
and $A^{\pm}(h)\subset\bP\setminus B^{\pm}(g)$. If moreover, $d(A^{\pm}(g),B^{\pm}(h))>\epsilon$
and $d(A^{\pm}(h),B^{\pm}(g))>\epsilon$ $ $ they are called \emph{$\epsilon$-transversal.
}Clearly, any finite set of transversal elements are $\epsilon$-transversal
for some $\epsilon>0$.

\section{Profinite completions and the congruence subgroup property\label{sec:Profinite-completions}}

\subsection{Basic properties\label{sub:Basic-properties-profinite}}

For background on profinite groups the reader can consult the book
of Ribes and Zalesskii \cite{RZ2010}. For completeness, we record
here the definition of the profinite completion of a group. We write
$A<_{fi}B$ to mean that $A$ is a finite-index subgroup of $B$.
Let $\Gamma$ be a finitely generated group. The profinite topology
on $\Gamma$ is defined by taking as a fundamental system of neighborhoods
of the identity the collection of all normal subgroups $N$ of $\Gamma$
such that $\Gamma/N$ is finite. One can easily show that a subgroup
$H$ is open in $\Gamma$ if and only if $H<_{fi}\Gamma$ . We can
complete $\Gamma$ with respect to this topology to get 
\[
\widehat{\Gamma}:=\varprojlim\{\Gamma/N:N\vartriangleleft_{fi}\Gamma\};
\]
this is the profinite completion of $\Gamma$. There is a natural
homomorphism $i:\Gamma\rightarrow\widehat{\Gamma}$ given by $i(\gamma)=\varprojlim(\gamma N)$.
A subgroup $\Lambda<\Gamma$ is profinitely dense in $\Gamma$ if
and only if $i(\Lambda)$ is dense in $\widehat{\Gamma}$ which in
turn is equivalent to the property that $\Lambda$ is mapped onto
every finite quotient of $\Gamma$. The minimal number of elements
of $\widehat{\Gamma}$ needed to generate a dense subgroup of $\widehat{\Gamma}$
is denoted by $d(\widehat{\Gamma})$.
\begin{lem}
\label{lem:subgroups of profinite completion} Let $\Gamma$ be a
residually finite group. There is a one-to-one correspondence between
the set $\cX$ of all cosets of finite-index subgroup of $\Gamma$
and the set $\cY$ of all cosets of open subgroups of $\widehat{\Gamma}$,
given by 
\[
X\mapsto cl(X)\quad(X\in\cX)
\]
 
\[
Y\mapsto Y\cap\Gamma\quad(Y\in\cY)
\]
 where $cl(X)$ denotes the closure of $X$ in $\widehat{\Gamma}$.
Moreover, this correspondence maps normal subgroups to normal subgroups.\end{lem}
\begin{proof}
See \cite[Window: profinite groups, Proposition 16.4.3]{LubotzkySegal}
for a proof of a correspondence between finite-index subgroup of $\Gamma$
and the set $\cY$ of all open subgroups of $\widehat{\Gamma}$. The
natural generalization to cosets is immediate.
\end{proof}

\subsection{The congruence subgroup property\label{sub:CSP}}

We give here a brief survey; for proofs of the assertions below the
reader can consult \cite{PR08}. 

For arithmetic groups, which are by definition commensurable to groups
of the form $G(\cO_{k})$ we can give another interesting topology
which is weaker then the profinite topology. As stated in the beginning,
$G$ is a simply-connected semisimple algebraic group defined over
a number field $k$ and equipped with an implicit $k$-embedding into
$GL_{n}$. For any non-zero ideal $I$ of $\cO_{k}$ let $K_{I}$
be the kernel of the map 
\begin{equation}
G(\cO_{k})\ra G(\cO_{k}/I).\label{eq:Cong map}
\end{equation}
The completion of $G(\cO_{k})$, with respect to the topology in which
these kernels form a system of neighborhoods of the identity, is called
the congruence completion. We denote this completion by $\ov{G(\cO_{k})}$.
Under the assumption we made on $G$, the strong approximation Theorem
holds for $G$. This means that the maps in \ref{eq:Cong map} are
surjective for all but finitely many ideals $I$ (See \cite[Theorem 7.12]{PR94}).
Using this, one can show that $\overline{G(\cO_{k,S})}$ is naturally
isomorphic to the open compact subgroup $\prod_{v\in V^{f}}G(\cO_{v})$
of $G(\bA_{k,f})$ where $V^{f}$ denotes the finite valuations of
$k$ and $\bA_{k,f}$ denotes the ring of finite Adeles (See \cite[section 1.2]{PR94}).
As the profinite topology is stronger, we have following the short
exact sequence 
\[
1\ra C(G)\ra\widehat{G(\cO_{k})}\ra\ov{G(\cO_{k})}\ra1
\]
and $C(G)$ is called the congruence kernel of $G$ w.r.t. $S$. If
$C(G)$ is finite, we say that $G$ admits the congruence subgroup
property. 

Finally, if $\Gamma$ is commensurable to $G(\cO_{k})$ then the completion
of $\Gamma$ w.r.t. the family $\{K_{I}\cap\Gamma\}_{I\vartriangleleft\cO_{k}}$
is called the congruence completion of $\Gamma$ and is denoted by
$\ov{\Gamma}$. One has a similar short exact sequence

\[
1\ra C(\Gamma)\ra\widehat{\Gamma}\ra\ov{\Gamma}\ra1.
\]
and one say that $\Gamma$ has the congruence subgroup property if
$|C(\Gamma)|<\infty$.

Assuming $\Gamma$ has the congruence subgroup property, the closure
$\overline{\Lambda}<\widehat{\Gamma}$ of a subgroup $\Lambda<\Gamma$
has finite-index if and only if for all but finitely many $ $ $I\vartriangleleft\cO_{k}$,
$\Lambda$ maps onto $\Gamma/(K_{I}\cap\Gamma)$. As stated above
By \cite{Weis84,PR94} every Zariski dense subgroup of $G(\bar{k})$
satisfy the last condition.

\section{Elements in cosets with prescribed dynamics\label{sec:Elements-in-cosets}}

In \cite[Theorem 3]{Tits1972}, Tits construct a strongly irreducible
representation (i.e. no finite union of proper subspaces is invariant)
$\rho:G(k)\rightarrow GL_{d}(K)$ for some $d\in\bN$ and some local
field $K$. From now on, we identify $G$ (also topologically) with
its image under $\rho$. We let $\bP=\bP(K^{d})$ and for $p\in\bP$
we write $gp:=\rho(g)p$. We begin by several lemmata:
\begin{lem}
\label{lem:Wg^nW lemma}Let $g_{0}\in\mathfrak{H}(G)$, and for $j=+,-$
set 
\[
X_{1}^{j}(\epsilon)=\{g\in\mathfrak{H}(G):d(A^{j}(g_{0}),A^{j}(g))<\epsilon\},
\]
\[
X_{2}^{j}(\epsilon)=\{g\in\mathfrak{H}(G):d_{h}(B^{j}(g_{0}),B^{j}(g))<\epsilon\}.
\]
Then, there exists $\epsilon>0$ such that for any $i=1,2,j=+,-$
and $h\in X_{i}^{j}(\epsilon)$, there exist a symmetric neighborhood
of the identity $W\subset G$ and $N\in\bN$ such that for all $n>N$
we have $Wh^{n}W\subset X_{i}^{j}(\epsilon)$.\end{lem}
\begin{proof}
We start with i=1 and $j=+$. We first choose $\epsilon$ such that
\[
\bigcup_{g\in X_{1}^{+}(\epsilon)}B^{+}(g)
\]
and $B(A^{+}(g_{0}),\epsilon)$ (the ball of radius $\epsilon$) are
disjoint and denote $X_{1}^{+}=X_{1}^{+}(\epsilon)$. Let $h\in X_{1}^{+}$
be an arbitrary element. We can choose $U\subset B(A^{+}(g_{0}),\epsilon)$
a neighborhood of $A^{+}(g_{0})$, and $W$ a neighborhood of the
identity in $G$, such that $WU=\{wu:w\in W,u\in U\}\subset B(A^{+}(g_{0}),\epsilon)$
and $A^{+}(h)\in U'$. By choosing $W$ even smaller we can also find
$U'\subset U$ such that the following are satisfied:
\begin{align*}
 & A^{+}(h)\in U',\\
 & WU'=\{wu:w\in W,u\in U'\}\subset U,\\
 & WU\subset\bP\setminus B^{+}(h),\\
 & \forall n\quad Wh^{n}W\subset\mathfrak{H}(G).
\end{align*}
Then, there exists $N\in\bN$ such that $h^{n}(WU)\subset U'$ for
all $n>N$ and it follows that $(Wh^{n}W)U\subset U$. Therefore,
any element $\tilde{h}\in Wh^{n}W$ with $n>N$, has a fixed point
inside $U$ which is necessarily $A^{+}(\tilde{h})$. Thus, $A^{+}(\tilde{h})\in U$
so by the choice of $U$ we have $d(A^{+}(\tilde{h}),A^{+}(g_{0}))<\epsilon$,
which implies that $Wh^{n}W\subset X_{1}^{+}$ for all $n>N$. Using
equation \ref{eq:inverse}, a proof along the exact same lines shows
the case $i=1,j=-$. 

For the case $i=2,$ one can use duality of hyperplanes and points
in the projective space and proceed with the same proof as above.
Alternatively, we can apply the same proof with the natural action
on $\bP(\wedge^{d-1}K^{d})$ where hyperplanes of the form $B^{+}(h)$
for some $h\in\mathfrak{H}(G)$ correspond to points. \end{proof}
\begin{lem}
\label{lem:adding arbitrary trans elt}Let $g_{1},\dots,g_{n}$ be
hyperbolic elements of $G$. There exists $g\in\mathfrak{H}(G)$ with
$g\perp g_{i}$  for all $i=1,\dots n$.\end{lem}
\begin{proof}
Let $U^{+}:=\{g\in G:gA^{+}(g_{1})\notin B^{\pm}(g_{j}),\forall j=1,\dots,n\}$;
We claim that $U^{+}$ is a non-empty Zariski open set. Assume not,
then $\cup_{i=1}^{n}B^{\pm}(g_{i})$ will contain an invariant set
which is a union of proper subspaces, which contradicts our assumption.
Similarly, $U^{-}:=\{g\in G:g^{-1}A^{-}(g_{j})\notin B^{\pm}(g_{1}),\forall j=1,\dots,n\}$
is a non-empty Zariski open set. As $G$ is Zariski connected, $U=U^{+}\cap U^{-}$
is non-empty open set and for any $g\in U$ we have$A^{\pm}(gg_{1}g^{-1})\notin B^{\pm}(g_{j})$
for all $j=1,\dots n$. Similar argument shows that there exists an
open $V\subset G$ such that for any $g\in V$ we have $g^{-1}A^{\pm}(g_{j})\notin B^{\pm}(g_{1})\Leftrightarrow A^{\pm}(g_{j})\notin B^{\pm}(gg_{1}g^{-1})$.
Thus, any element of the form $gg_{1}g^{-1}$ for some $g\in U\cap V$
is transversal to all $g_{j}$, as needed. 
\end{proof}
The aim of this section is to find such an element $g$ in a given
coset of $G/\Lambda$ where $\Lambda$ is a lattice in $G$. To this
end, we begin by {}``approximating'' the transversal element we
get from Lemma \ref{lem:adding arbitrary trans elt}.
\begin{lem}
\label{lem:W neigh}Given $\epsilon>0$, assume that $\{g_{1},\dots,g_{s},g\}$
is a set of pairwise $\epsilon$-transversal hyperbolic elements.
There exist an integer $N=N(g,\epsilon)$ and a symmetric neighborhood
$W=W(g,\epsilon)$ of the identity in $G$ such that if $n>N$ then
for any $h\in Wg^{n}W$, $\{g_{1},\dots,g_{s},h\}$ is a set of pairwise
$\frac{\epsilon}{2}$-transversal hyperbolic elements.\end{lem}
\begin{proof}
We first claim that any $h\in\mathfrak{H}(G)$ with
\begin{equation}
\begin{alignedat}{1}d(A^{+}(h),A^{+}(g))<\frac{\epsilon}{2},\quad & d(A^{-}(h),A^{-}(g))<\frac{\epsilon}{2}\\
d_{h}(B^{+}(h),B^{+}(g))<\frac{\epsilon}{2},\quad & d_{h}(B^{-}(h),B^{-}(g))<\frac{\epsilon}{2}
\end{alignedat}
\label{eq:close to g}
\end{equation}
is $\frac{\epsilon}{2}$- transversal to any element which is $\epsilon$-transversal
to $g$. Indeed, say $g'$ is $\epsilon$-transversal to $g$, then
we have for example: 
\begin{align*}
\epsilon<d(A^{+}(g),B^{\pm}(g'))\leq d(A^{+}(g),A^{+}(h))+d(A^{+}(h),B^{\pm}(g'))\\
\epsilon<d(A^{\pm}(g'),B^{+}(h))\leq d(A^{+}(g'),B^{+}(h))+d_{h}(B^{+}(h),B^{+}(g))
\end{align*}
so using \ref{eq:close to g} we see that $g'$ and $h$ are $\frac{\epsilon}{2}$-transversal.

It follows readily from Lemma \ref{lem:Wg^nW lemma} that there exists
a symmetric neighborhood $W$ of the identity of $G$ and $N=N(g,\epsilon)\in\bN$
such that for large enough $n>N$, any $h\in Wg^{n}W$ is hyperbolic
and satisfy \ref{eq:close to g}.
\end{proof}
Before proceeding to the main proposition of this section we recall
some notions from Ergodic Theory. By the Borel Harish-Chandra Theorem
(See \cite[\S 4.6]{PR94}), arithmetic subgroups are lattices. By
definition, $\Lambda$ is a lattice in $G$ if $\Lambda$ is discrete
and $G/\Lambda$ carry a finite $G$-invariant measure which we denote
by $\mu$. The action of $G$ on $G/\Lambda$ allows us to use techniques
from Ergodic Theory. In particular, for semisimple groups $G$ as
in our case, we have the vanishing theorem of Howe-Moore (see \cite[Chaper III]{BM2000}
and the references therein). It states that for any $g\in G$ with
the property that for any compact subset $K\subset G$ there exists
$M$ such that $g^{n}\notin K$ for all $n>M$, we have 
\begin{equation}
\lim_{n\rightarrow\infty}\mu(A\cap g^{n}B)\rightarrow\mu(A)\mu(B)\label{eq:mixing property}
\end{equation}
for any measurable sets $A,B\subset G/\Lambda$.
\begin{prop}
\label{prop:trans elmt in coset}Let $\Lambda$ be a lattice in $G$
, $\{g_{1},\dots,g_{s}\}$ a set of pairwise transversal hyperbolic
elements, and $x\in G$. Then there exists $h\in x\Lambda$ such that
$\{g_{1},\dots,g_{s},h\}$ is a set of pairwise transversal hyperbolic
elements.\end{prop}
\begin{proof}
Using Lemma \ref{lem:adding arbitrary trans elt} we find some $g\in\mathfrak{H}(G)$
such that $\{g_{1},\dots,g_{s},g\}$ is a set of pairwise $\epsilon_{0}$-transversal
hyperbolic elements for some $\epsilon_{0}>0$. In order to finish
the proof, we only need to find $h\in x\Lambda$ which also satisfy
$h\in Wg^{n}W$ for some $n>N$ where $W=W(g,\epsilon_{0}),N=N(g,\epsilon_{0})$
are supplied by Lemma \ref{lem:W neigh}. To this end, we consider
the homogeneous space $G/\Lambda$ equipped with the unique left invariant
probability measure $\mu$ on $G/\Lambda$. As $g\in\mathfrak{H}(G)$,
for any compact subset $K\subset G$ there exists $M$ such that $g^{n}\notin K$
for all $n>M$. Therefore, by the Howe-Moore Theorem stated above,
we have that 
\[
\lim_{n\rightarrow\infty}\mu(g^{n}W\Lambda\cap Wx\Lambda)=\mu(W\Lambda)\mu(Wx\Lambda)>0
\]
so for large enough $n$, and in particular for some $n>N$, $g^{n}W\Lambda\cap Wx\Lambda\neq\emptyset$.
This yields some $w_{1},w_{2}\in W,\gamma\in\Lambda$ with $g^{n}w_{1}=w_{2}x\gamma$
so $h:=x\gamma=w_{2}^{-1}g^{n}w_{1}$, as desired.
\end{proof}

\section{Free $g_{0}$-rooted systems\label{sec:Free-rooted-systems}}
\begin{defn}
Given $g_{o}\in\mathfrak{H}(G)$, a tuple $\{g_{i}\}_{i=1}^{s}\subset\mathfrak{H}(G)$
is called a $g_{0}$-rooted free system if there exist open sets $\{X_{i}\}_{i=0}^{s}$
of $\bP$ satisfying: \end{defn}
\begin{enumerate}
\item $\{X_{i}\}_{i=0}^{s}$ are pairwise disjoint,
\item $A^{\pm}(g_{i})\subseteq X_{i}\subset\overline{X_{i}}\subset\bP\setminus B^{\pm}(g_{0})$
for all $i=0,.\dots s$,
\item $A^{\pm}(g_{0})\subseteq X_{0}\subset\overline{X_{0}}\subset\bP\setminus\cup_{i=1}^{s}B^{\pm}(g_{i})$
,
\item $g_{i}(\overline{X_{j}})\subset X_{i}$ for any $i,j\in\{0,\dots,s\}$.
\end{enumerate}
Note that by the so-called ping-pong Lemma \cite[Proposition 1.1]{Tits1972},
the elements $g_{0},\dots,g_{s}$ are independent. $(\{g_{i}\}_{i=1}^{s},g_{0})$
is clearly free. The following lemma exemplify how one can use $g_{0}$-rooted
systems for enlarge a given independent set of elements. 
\begin{lem}
\label{lem:Adding element with infinite occurences}Assume we are
given a $g_{0}$-rooted free system $\{g_{1},\dots,g_{s}\}$, a subset
$F\subset G$ and and element $h\in F\cap\mathfrak{H}(G)$ such that
\begin{itemize}
\item the set $M_{1}=\{n:g_{0}^{n}hg_{0}^{-n}\in F\}$ is infinite,
\item the sets $M_{2,l}=\{n:g_{0}^{l}h^{n}g_{0}^{-l}\in F\}$ are infinite
whenever $l\in M_{1}$,
\item $h\perp g_{i}$ for $i=0,\dots s$. 
\end{itemize}
Then, there exist $k,n_{0},n_{1}\in\bN$ such that $g_{s+1}:=g_{0}^{n_{0}}h^{n_{1}}g_{0}^{-n_{0}}\in F$
and $\{g_{1},\dots,g_{s},g_{s+1}\}$ is a $g_{0}^{k}$-rooted free
system.\end{lem}
\begin{proof}
Let $\{X_{i}\}_{i=0}^{s}$ be the sets showing that $\{g_{1},\dots,g_{s}\}$
is a $g_{0}$-rooted free system and let $h_{n}:=g_{0}^{n}hg_{0}^{-n}$.
As $g_{0}\perp h$, there exists $N_{1}\in\bN$ such that for any
$n>N_{1}$, the following holds: 
\begin{equation}
A^{\pm}(h_{n})=g_{0}^{n}(A^{\pm}(h))\subset X_{0}\label{eq:in X0-1}
\end{equation}
and for any $i=1,\dots,s$ we have $g_{0}^{-n}\overline{X_{i}}\subset\bP\setminus B^{\pm}(h)$
which is equivalent to 
\begin{equation}
\overline{X_{i}}\subset\bP\setminus g_{0}^{n}B^{\pm}(h)=\bP\setminus B^{\pm}(h_{n}).\label{eq: trans to i-1}
\end{equation}
This show that $\{g_{1},\dots,g_{s},h_{n}\}$ are pairwise transversal
for all $n>N_{1}$. Furthermore, it is easy to see that $g_{0}$ and
$h_{n}$ are transversal for any $n\in\bN$. By assumption, there
are arbitrarily large $n\in\bN$ for which $h_{n}\in F$ so we can
choose $n_{0}>N_{1}$ such that $h_{n_{0}}\in F$.

We now define open subsets 
\begin{equation}
\{Y_{i}\}_{i=0}^{s+1}\label{eq:Y_i's}
\end{equation}
 that will show that for some $n_{1}$, $\{g_{1},\dots,g_{s},h_{n_{0}}^{n_{1}}\}$
is $g_{0}^{k}$-rooted system ($n_{1}$ and $k$ will be defined momentarily).
Let $Y_{i}=X_{i}$ for $i=1,\dots,s$ and $Y_{s+1}$ be an open subset
of $\bP$ such that $Y_{s+1}\subset\overline{X_{0}}$ and 
\begin{equation}
A^{\pm}(h_{n_{0}})\subset Y_{s+1}\subset\overline{Y_{s+1}}\subset\bP\setminus B^{\pm}(g_{0}).\label{eq:n0 away from g0-1}
\end{equation}
Furthermore, we let $Y_{0}$ be an open subset of $\bP$ such that
\begin{equation}
A^{\pm}(g_{0})\subset Y_{0}\subset\overline{Y_{0}}\subset X_{0}\,\text{ and }\,\overline{Y_{0}}\subset\bP\setminus B^{\pm}(h_{n_{0}}).\label{eq:def of Y0-1}
\end{equation}
Such sets exist since we've seen that $\{g_{0},\dots,g_{s},h_{n_{0}}\}$
are pairwise transversal; this also imply that there exists $N_{2}$
such that for any integer $n$ with $|n|>N_{2}$ we have that 
\begin{equation}
h_{n_{0}}^{n}(Y_{i})\subset Y_{s+1}\,\,\text{for }i=0,\dots,s.\label{eq: large power to Xs-1}
\end{equation}
Finally, by assumption , the set 
\begin{equation}
M:=\{n\in\bN:h_{n_{0}}^{n}\in F\}\label{eq: infinite set of powers in the coset-1}
\end{equation}
is infinite. Let $g_{s+1}=h_{n_{0}}^{n_{1}}$ for some $n_{1}\in M$
with $n_{1}>N_{2}$. The above transversality also imply that there
exists $k\in\bN$ such that 
\begin{equation}
g_{0}^{k}(\cup_{i=0}^{s+1}Y_{i})\subset Y_{0}.\label{eq:large enough k-1}
\end{equation}

We end the proof by showing that $\{g_{1},\dots,g_{s},g_{s+1}\}$
is $g_{0}^{k}$-rooted system, using the sets $\{Y_{i}\}_{i=0}^{s+1}$.
Condition (1) is satisfied by the choice of the sets $\{Y_{i}\}_{i=0}^{s+1}$.
$ $Using equation \ref{eq:attractig same for powers}, we see that
equation \ref{eq:n0 away from g0-1} and the fact that $X_{i}=Y_{i}$
for $i=1,\dots,s$ imply the Condition (2). Similarly, equation \ref{eq:def of Y0-1}
imply Condition (3). Lastly, equation \ref{eq: large power to Xs-1}
imply condition (4) for i=s+1 and equation \ref{eq:large enough k-1}
does it for i=0.
\end{proof}
The following two propositions are the main ingredients in the proofs
of our main Theorems:
\begin{prop}
\label{prop:adding elt in coset to free system}Let $\Gamma_{1}\lhd_{fi}\Gamma$,
$x\in\Gamma$ and $\{g_{1},\dots,g_{s}\}$ be a $g_{0}$-rooted free
system. Then, there exist $k\in\bN$ and a hyperbolic element $g_{s+1}\in x\Gamma_{1}$
such that $\{g_{1},\dots,g_{s},g_{s+1}\}$ is a $g_{0}^{k}$-rooted
free system.\end{prop}
\begin{proof}
By Proposition \ref{prop:trans elmt in coset} we can find hyperbolic
$h\in x\Gamma_{1}$ such that $\{g_{1},\dots,g_{s},h$\} are pairwise
transversal. Moreover, as $g_{0}\in\Gamma$ and $\Gamma_{1}\lhd_{fi}\Gamma$,
the set $M_{1}=\{n:g_{0}^{n}hg_{0}^{-n}\in x\Gamma_{1}\}$ is infinite
and so are the sets $M_{2,l}=\{n:g_{0}^{l}h^{n}g_{0}^{-l}\in x\Gamma_{1}\}$
whenever $l\in M_{1}$. Thus, applying Lemma \ref{lem:Adding element with infinite occurences}
with $F=x\Gamma_{1}$ we find $k\in\bN$ and a hyperbolic element
$g_{s+1}\in x\Gamma_{1}$ such that $\{g_{1},\dots,g_{s},g_{s+1}\}$
is a $g_{0}^{k}$-rooted free system, as claimed.\end{proof}
\begin{prop}
\label{prop:adding elt outside a subgroup}Let $H<\Gamma$, be a profinitely
dense subgroup and $\{g_{1},\dots,g_{s}\}$ be a $g_{0}$-rooted free
system, with $\{g_{i}\}_{i=0}^{s}\subset\Gamma$. Then, there exists
a hyperbolic element $g_{s+1}\notin H$ such that $\{g_{1},\dots,g_{s},g_{s+1}\}$
is a $\tilde{g_{0}}$-rooted free system for some $\tilde{g_{0}}\in\mathfrak{H}(G)\cap\Gamma$
(typically $\tilde{g_{0}}=g_{0}^{k}$ for some $k\in\bN$).\end{prop}
\begin{proof}
Consider first the case when $g_{0}\in H$. By Proposition 3 of \cite{MS81}
we know that $H$ is Zariski dense since it is profinitely dense.
It follows from \cite[Lemma 8]{MS81} that there exists $\tilde{h}\in\mathfrak{H}(G)$
with $\tilde{h}\in\Gamma\setminus H$. Now, 
\[
W=\{w\in G:\forall i\in\{0,\dots s\},(w\tilde{h}w^{-1})\perp g_{i}\}
\]
 is Zariski open; indeed,
\[
W=\cap_{i=1}^{s}\{w:(wA^{\pm}(\tilde{h}))\cap B^{\pm}(g_{i})=\emptyset\}\bigcap\cap_{i=1}^{s}\{w:B^{\pm}(\tilde{h})\cap w^{-1}A^{\pm}(g_{i})=\emptyset\}
\]
and each on of the sets in the intersection is clearly a Zariski open
set. As $H$ is Zariski dense, we can find some element $w\in H\cap W$.
Let $h=w\tilde{h}w^{-1}$. Clearly, $h\in\mathfrak{H}(G)$, $h\in\Gamma\setminus H$
and the elements of $\{g_{1},\dots,g_{s},h$\} are pairwise transversal.
Moreover, as $g_{0}\in H$, the set $M_{1}=\{n:g_{0}^{n}hg_{0}^{-n}\in\Gamma\setminus H\}=\bN$
and the sets $M_{2,l}=\{n:g_{0}^{l}h^{n}g_{0}^{-l}\in\Gamma\setminus H\}$
are infinite whenever $l\in M_{1}$. Thus, applying Lemma \ref{lem:Adding element with infinite occurences}
with $F=\Gamma\setminus H$ we find $k\in\bN$ and a hyperbolic element
$g_{s+1}\in\Gamma\setminus H$ such that $\{g_{1},\dots,g_{s},g_{s+1}\}$
is a $g_{0}^{k}$-rooted free system, which conclude the case when
$g_{0}\in H$.

Now assume $g_{0}\notin H$. Let $\{X_{i}\}_{i=0}^{s}$ be the sets
showing that $\{g_{1},\dots,g_{s}\}$ is a $g_{0}$-rooted free system
and set $d=min_{1\leq i\leq s}d(B^{\pm}(g_{0},\overline{X_{i}}))$.
Let
\[
U_{1}=\{g\in G:g\in\mathfrak{H}(G),g\perp g_{0}\},
\]
\[
U_{2}=\{g\in G:A^{\pm}(g)\subset X_{0},d_{h}(B^{\pm}(g_{0}),B^{\pm}(g))<\frac{d}{4}\}
\]
and $U=U_{1}\cap U_{2}$. We now consider two possibilities: if $U\cap H\neq\emptyset$
it is easy to see that there exist $h_{0}\in U\cap H$ and $k\in\bN$
such that $\{g_{1},\dots,g_{s}\}$ is a $h_{0}^{k}$-rooted free system.
Thus, this possibility reduces to the first case with $g_{0}$ interchanged
with $h_{0}^{k}\in H$. 

Therefore we assume that $U\cap H=\emptyset$. Note that by Lemma
\ref{lem:Wg^nW lemma} for any $g\in U$ there exist $W$ and $N_{0}\in\bN$
such that $Wg^{n}W\subset U$ for all $n>N$. As $g\in\mathfrak{H}(G)$,
for any compact subset $K\subset G$ there exists $M$ such that $g^{n}\notin K$
for all $n>M$. Therefore, by the Howe-Moore Theorem stated above,
we have that 
\[
\lim_{n\rightarrow\infty}\mu(g^{n}W\Gamma\cap W\Gamma)=\mu(W\Gamma)\mu(W\Gamma)>0
\]
so for large enough $n$, and in particular for some $n>N_{0}$, $g^{n}W\Gamma\cap W\Gamma\neq\emptyset$.
This yields some $w_{1},w_{2}\in W,\gamma\in\Gamma$ with $g^{n}w_{1}=w_{2}\gamma$
so $h:=\gamma=w_{2}^{-1}g^{n}w_{1}\in\Gamma\cap U$ as desired. By
the definition of $U$, the elements of $\{g_{1},\dots,g_{s},h\}$
are pairwise transversal, and all of them are transversal to $g_{0}$.
To show that they form a $g_{0}$-rooted free system, we can follows
the proof of Lemma \ref{lem:Adding element with infinite occurences}
from equation \ref{eq:Y_i's} onwards, interchanging $h_{n_{0}}$
with $h$ everywhere, and noting that \ref{eq: infinite set of powers in the coset-1}
is also true in our situation as the set 
\[
M=\{n:h^{n}\notin H\}
\]
 is infinite since $h\notin H$ by the assumption that $U\cap H=\emptyset$.
\end{proof}

\section{Proof of theorems \ref{thm:Main Theorem} and \ref{thm:Main Theorem maximal subgroups}\label{sec:Proof-main-thm}}
\begin{proof}[Proof of Theorem \ref{thm:Main Theorem}:]
We first claim that there exist $h_{1},h_{2},g_{0}\in\Gamma$ such
that $\{h_{1},h_{2}\}$ is $g_{0}$-rooted free system and $\langle h_{1},h_{2}\rangle$
is Zariski dense in $G$. Indeed, by \cite[Theorem 3]{Tits1972} we
can find $f_{1},f_{2}\in\Gamma$ such that $\langle f_{1},f_{2}\rangle$
is free and Zariski dense in $G$. By connectedness of $G$, the same
is true for $f_{1}^{k},f_{2}^{k}$ for any $k\in\bN$. Moreover, by
Proposition \ref{prop:trans elmt in coset} we can find $f_{0}\in\Gamma$
such that $\{f_{0},f_{1},f_{2}\}$ are pairwise transversal. Therefore,
there exists a $k\in\bN$ such that $\{f_{1}^{k},f_{2}^{k}\}$ are
$f_{0}^{k}$-free rooted system. Thus the claim is proved by letting
$g_{0}:=f_{0}^{k},h_{1}:=f_{1}^{k},h_{2}:=f_{2}^{k}$. 

As explain at the end of section \ref{sub:CSP}, since $\Gamma$ admits
that Congruence Subgroup Property and $\langle h_{1},h_{2}\rangle$
is Zariski dense, the closure of $\langle h_{1},h_{2}\rangle$ is
of finite-index in $\widehat{\Gamma}$. By Lemma \ref{lem:subgroups of profinite completion}
there exists $\Gamma_{1}\vartriangleleft_{fi}\Gamma$ such that $\widehat{\Gamma_{1}}\subset\overline{\langle h_{1},h_{2}\rangle}$.
Let $y_{1},\dots,y_{m}$ be elements of $\widehat{\Gamma}$ which
generate of $\widehat{\Gamma}/\widehat{\Gamma_{1}}$ with $m\leq d(\widehat{\Gamma})$.
By Lemma \ref{lem:subgroups of profinite completion}, there exist
$x_{1},\dots,x_{m}$ with $x_{i}\Gamma_{1}=y_{i}\widehat{\Gamma_{1}}\cap\Gamma$.
Then, using Proposition \ref{prop:adding elt in coset to free system}
inductively, we find $\{g_{i}\}_{i=1}^{m}\in\Gamma$ such that $g_{i}\in x_{i}\Gamma_{2}$
and $\{h_{1},h_{2},g_{1},\dots,g_{m}\}$ is a $g_{0}^{k}$-rooted
free system for some $k\in\bN$. Therefore the image of $\Gamma_{1}$
in $\widehat{\Gamma}$ together with $\{g_{1},\dots,g_{m}\}$ generate
$\widehat{\Gamma}$ (topologically). This shows that that $\langle h_{1},h_{2},g_{1},\dots,g_{m}\rangle$
is profinitely dense and is free of rank $m+2$, as asserted.
\end{proof}

\begin{proof}[Proof of Theorem \ref{thm:Main Theorem maximal subgroups}]
 Assume, by a way of contradiction, that $\mathfrak{U_{m}}$ is countable
and let $\{U_{i}\}_{i=1}^{\infty}$ be some enumeration of it. Let
$\{F_{i}\}_{i=1}^{\infty}$ be an enumeration of all the cosets of
finite-index subgroups of $\Gamma$. Let $g_{0}\in\Gamma$ be some
hyperbolic element. Using Proposition \ref{prop:adding elt in coset to free system}
(with the empty $g_{0}$-rooted free system) we can find $g_{1}\in F_{1}$
and $g_{0,1}\in\mathfrak{H}(G)\cap\Gamma$ such that $\{g_{1}\}$
is a $g_{0,1}$-rooted free system. Similarly, using Proposition \ref{prop:adding elt outside a subgroup},
we can find $g_{2}$ and $g_{0,2}\in\mathfrak{H}(G)\cap\Gamma$ such
that $g_{2}\notin U_{1}$ and $\{g_{1},g_{2}\}$ is a $g_{0,2}$-rooted
free system. Continuing in this fashion and using Propositions \ref{prop:adding elt in coset to free system}
and \ref{prop:adding elt outside a subgroup} alternately, we find
a sequences $\{g_{i}\}_{i=1}^{\infty},\{g_{0,i}\}_{i=1}^{\infty}\subset\mathfrak{H}(G)\cap\Gamma$
such that for any $k\in\bN$, $g_{2k}\notin U_{k}$ , $g_{2k-1}\in F_{k}$
and for any $l\in\bN$, $\{g_{1},\dots,g_{l}\}$ is $g_{0,l}$-rooted
free system. Let $H=\langle\{g_{i}\}_{i=1}^{\infty}\rangle$; we claim
that $H$ is free and profinitely dense. Indeed, $H$ is freely generated
by $ $$\{g_{i}\}_{i=1}^{\infty}$ since any finite subset of $\{g_{i}\}_{i=1}^{\infty}$
is contained in some free rooted system. It is profinitely dense as
otherwise $H$ would be contained in a finite-index subgroup $\Lambda$,
which is impossible as $H$ has elements in each coset of $\Lambda$
in $\Gamma$. Therefore $H\subseteq U_{i}$ for some $i$. This is
a contradiction as $g_{2i}\notin U_{i}$.
\end{proof}

\subsection*{Acknowledgments }

We acknowledge the support of the ERC grant 226135, the ERC Grant
203418, the ISEF foundation, the Ilan and Asaf Ramon memorial foundation,
the \textquotedbl{}Hoffman Leadership and Responsibility\textquotedbl{}
fellowship program, the ISF grant 1003/11 and the BSF grant 2010295.

\author{\bibliographystyle{plain}
\bibliography{biblio}

\def\cprime{$'$} \def\cprime{$'$} \def\cprime{$'$}
\begin{thebibliography}{10}

\bibitem{BM2000}
M.~Bachir Bekka and Matthias Mayer.
\newblock {\em Ergodic theory and topological dynamics of group actions on
  homogeneous spaces}, volume 269 of {\em London Mathematical Society Lecture
  Note Series}.
\newblock Cambridge University Press, Cambridge, 2000.

\bibitem{BG2003}
E.~Breuillard and T.~Gelander.
\newblock On dense free subgroups of {L}ie groups.
\newblock {\em J. Algebra}, 261(2):448--467, 2003.

\bibitem{BG2007}
E.~Breuillard and T.~Gelander.
\newblock A topological {T}its alternative.
\newblock {\em Ann. of Math. (2)}, 166(2):427--474, 2007.

\bibitem{LS76}
Ronnie Lee and R.~H. Szczarba.
\newblock On the homology and cohomology of congruence subgroups.
\newblock {\em Invent. Math.}, 33(1):15--53, 1976.

\bibitem{LubotzkySegal}
Alexander Lubotzky and Dan Segal.
\newblock {\em Subgroup growth}, volume 212 of {\em Progress in Mathematics}.
\newblock Birkh\"auser Verlag, Basel, 2003.

\bibitem{MS81}
G.~A. Margulis and G.~A. So{\u\i}fer.
\newblock Maximal subgroups of infinite index in finitely generated linear
  groups.
\newblock {\em J. Algebra}, 69(1):1--23, 1981.

\bibitem{Neukirch}
J{\"u}rgen Neukirch.
\newblock {\em Algebraic number theory}, volume 322 of {\em Grundlehren der
  Mathematischen Wissenschaften [Fundamental Principles of Mathematical
  Sciences]}.
\newblock Springer-Verlag, Berlin, 1999.
\newblock Translated from the 1992 German original and with a note by Norbert
  Schappacher, With a foreword by G. Harder.

\bibitem{PR94}
Vladimir Platonov and Andrei Rapinchuk.
\newblock {\em Algebraic groups and number theory}, volume 139 of {\em Pure and
  Applied Mathematics}.
\newblock Academic Press Inc., Boston, MA, 1994.
\newblock Translated from the 1991 Russian original by Rachel Rowen.

\bibitem{PR08}
G.~{Prasad} and A.~S. {Rapinchuk}.
\newblock {Developments on the congruence subgroup problem after the work of
  Bass, Milnor and Serre}.
\newblock {\em ArXiv e-prints}, September 2008.

\bibitem{RZ2010}
Luis Ribes and Pavel Zalesskii.
\newblock {\em Profinite groups}, volume~40 of {\em Ergebnisse der Mathematik
  und ihrer Grenzgebiete. 3. Folge. A Series of Modern Surveys in Mathematics
  [Results in Mathematics and Related Areas. 3rd Series. A Series of Modern
  Surveys in Mathematics]}.
\newblock Springer-Verlag, Berlin, second edition, 2010.

\bibitem{SC78}
Peter Scott.
\newblock Subgroups of surface groups are almost geometric.
\newblock {\em J. London Math. Soc. (2)}, 17(3):555--565, 1978.

\bibitem{SV2000}
G.~A. Soifer and T.~N. Venkataramana.
\newblock Finitely generated profinitely dense free groups in higher rank
  semi-simple groups.
\newblock {\em Transform. Groups}, 5(1):93--100, 2000.

\bibitem{Tits1972}
J.~Tits.
\newblock Free subgroups in linear groups.
\newblock {\em J. Algebra}, 20:250--270, 1972.

\bibitem{Tr62}
Stanton~M. Trott.
\newblock A pair of generators for the unimodular group.
\newblock {\em Canad. Math. Bull.}, 5:245--252, 1962.

\bibitem{Weis84}
Boris Weisfeiler.
\newblock Strong approximation for {Z}ariski-dense subgroups of semisimple
  algebraic groups.
\newblock {\em Ann. of Math. (2)}, 120(2):271--315, 1984.

\end{thebibliography}
}
\end{document}